\documentclass[11pt,twoside]{amsart}
\usepackage{amsmath, amsthm, amscd, amsfonts, amssymb, graphicx, color}
\usepackage[bookmarksnumbered, colorlinks, plainpages]{hyperref}

\newtheorem{theorem}{Theorem}[section]
\newtheorem{lemma}[theorem]{Lemma}
\newtheorem{proposition}[theorem]{Proposition}
\newtheorem{corollary}[theorem]{Corollary}

\theoremstyle{definition}

\newtheorem{example}[theorem]{Example}

\newtheorem{remark}[theorem]{Remark}

\numberwithin{equation}{section}

\def\<{\langle}
\def\>{\rangle}

\long\def\alert#1{\smallskip{\hskip\parindent\vrule%
\vbox{\advance\hsize-2\parindent\hrule\smallskip\parindent.4\parindent%
\narrower\noindent#1\smallskip\hrule}\vrule\hfill}\smallskip}

\begin{document}

\baselineskip 17 pt

\title[\ci\ ON WEAKLY $n$-ABSORBING IDEALS OF COMMUTATIVE RINGS]{ON WEAKLY $n$-ABSORBING IDEALS OF COMMUTATIVE RINGS}

\author{Hojjat Mostafanasab}
\address{Department of Mathematics and Applications, University of Mohaghegh Ardabili, P. O. Box 179, Ardabil, Iran} \email{h.mostafanasab@uma.ac.ir, ~h.mostafanasab@gmail.com}

\author{Fatemeh Soheilnia}
\address{Department of Mathematics, South Tehran Branch, Islamic Azad University, Tehran, Iran 
} \email{soheilnia@gmail.com}

\author{Ahmad Yousefian Darani$^{*}$}
\address{Department of Mathematics and Applications, University of Mohaghegh Ardabili, P. O. Box 179, Ardabil, Iran} 
\email{yousefian@uma.ac.ir}

\thanks{{\scriptsize
\hskip -0.4 true cm 2010 {\it Mathematics Subject Classification}.
Primary 13A15; Secondary 13F05, 13G05.
}}
\thanks{$^*$Corresponding author\\
This paper has been accepted in ``Anal. Stii. Ale Univ. Al.I. Cuza Iasi Secti. Matem." on 2 Sep 2015.}

\begin{abstract}
All rings are commutative with $1\neq0$. The purpose of this paper is to investigate the concept of weakly $n$-absorbing 
ideals generalizing weakly $2$-absorbing ideals. We prove that over a $u$-ring $R$ the Anderson-Badawi's conjectures about
$n$-absorbing ideals and the Badawi-Yousefian's question about weakly $2$-absorbing ideals hold.
\end{abstract}

\maketitle

\section{Introduction}
Throughout this paper all rings are commutative with a nonzero identity. 
Recall from \cite{AS} that a proper ideal $I$ of a commutative ring $R$ 
is said to be a {\it weakly prime ideal of} $R$ if whenever $a,b\in R$ and 
$0\neq ab\in I$, then either $a\in I$ or $b\in I$.
Badawi in \cite{B} generalized the concept of prime ideals in a different
way. He defined a nonzero proper ideal $I$ of $R$ to be a {\it 2-absorbing
ideal} of $R$ if whenever $a, b, c\in R$ and $abc\in I$, then $ab\in I$ or $%
ac\in I$ or $bc\in I$. Anderson and Badawi 
\cite{AB1} generalized the concept of $2$-absorbing ideals to $n$-absorbing
ideals. According to their definition, a proper ideal $I$ of $R$ is called
an $n$-{\it absorbing} (resp. {\it strongly $n$-absorbing}) ideal if whenever $%
a_1\cdots a_{n+1}\in I$ for $a_1,\dots,a_{n+1}\in R$ (resp. $I_1\cdots
I_{n+1}\subseteq I$ for ideals $I_1,\dots, I_{n+1}$ of $R$), then there are 
$n$ of the $a_i$'s (resp. $n$ of the $I_i$'s) whose product is in $I$. Thus a strongly 1-absorbing ideal is just a prime ideal. Clearly a strongly $n$-absorbing ideal of $R$ is also an $n$-absorbing ideal of $R$. Anderson and Badawi conjectured that these
two concepts are equivalent, e.g., they proved that an ideal $I$ of a Pr\"{u}fer domain $R$ is strongly $n$-absorbing if and only if $I$ is an $n$-absorbing ideal of $R$, \cite[Corollary 6.9]{AB1}. They also gave several results relating strongly $n$-absorbing ideals.
The concept $2$-absorbing ideals has another generalization, called weakly $2$-absorbing
ideals, which has studied in \cite{YB}. A proper ideal $I$ of $R$ to be a
{\it weakly 2-absorbing ideal of} $R$ if whenever $a, b, c\in R$ and $0\neq
abc\in I$, then $ab\in I$ or $ac\in I$ or $bc\in I$. Generally, we say that a proper ideal $I$ of $R$ is called
a {\it weakly $n$-absorbing} (resp. {\it strongly weakly $n$-absorbing}) ideal if whenever $
0\neq a_1\cdots a_{n+1}\in I$ for $a_1,\dots,a_{n+1}\in R$ (resp. $0\neq I_1\cdots
I_{n+1}\subseteq I$ for ideals $I_1,\dots,I_{n+1}$ of $R$), then there are 
$n$ of the $a_i$'s (resp. $n$ of the $I_i$'s) whose product is in $I$. Clearly a strongly weakly $n$-absorbing ideal of $R$ is also a weakly $n$-absorbing ideal of $R$. In \cite{Q}, Quartararo et al. said that a commutative ring $R$ is a $u$-{\it ring} provided
$R$ has the property that an ideal contained in a finite union of ideals
must be contained in one of those ideals. They show that every B$\acute{\rm e}$zout ring is a $u$-ring. Moreover, they proved that 
every Pr\"{u}fer domain is a $u$-domain.

In section 2, we give some basic properties of weakly $n$-absorbing ideals. For example, we show that 
$I$ is a weakly $n$-absorbing ideal of an integral domain $R$ if and only if
$\langle I,X\rangle$ is a weakly $n$-absorbing ideal of $R[X]$. If $I$ is a secondary ideal of a ring $R$ and $J$ is a weakly $n$-absorbing ideal of $R$, then $I\cap J$ is secondary. Let $I$ be a weakly $n$-absorbing ideal of $R$ that is not an $n$-absorbing
ideal. Then $\sqrt{I}={\rm Nil}(R)$, and also if $w\in {\rm Nil}(R)$, then either $w^{n}\in I$ or $w^{n-i}I^{i+1}=\{0\}$ for every $0\leq i\leq n-1$.

In section 3, we prove that if $R$ is a ring and $n$ is a positive integer such that every proper ideal of
$R$ is a weakly $n$-absorbing ideal of $R$, then ${\rm dim}(R)=0$, $R$ has at most $n+1$ prime ideals that are pairwise comaximal,
and $Jac(R)^{n+1}=0$. Let $(R_1, M_1),\dots,(R_s, M_s)$ be quasi-local commutative rings and let $R= R_1\times\cdots\times R_s$. If every proper ideal of $R$ is a weakly $n$-absorbing ideal of $R$, then $M_1^n =M_2^n=\cdots= M_s^n = \{0\}$. Moreover we show that
every proper ideal of a decomposable commutative ring $R = R_1\times R_2\times \cdots \times R_{n+1}$ is a weakly $n$-absorbing ideal of $R$ if and only if all of $R_i$'s  are fields.

In section 4, we investigate the following conjectures of Anderson and Badawi \cite{AB1}:\\
$\mathbf{Conjecture~1.}$
{Let $n$ be a positive integer. Then a proper ideal $I$ of a ring $R$ is
a strongly $n$-absorbing ideal of $R$ if and only $I$ is an $n$-absorbing ideal of $R$.}\\
$\mathbf{Conjecture~2.}$
{Let $n$ be a positive integer, and let $I$ be an $n$-absorbing ideal of a
ring $R$. Then $(\sqrt{I})^{n}\subseteq I$.}\\
In \cite{AB1}, they proved that Conjecture 1 implies Conjecture 2. Also, they show that if $R$ is a B$\acute{\rm e}$zout ring and $I$ is an $n$-absorbing ideal of $R$ such that $\sqrt{I}$ is a prime ideal of $R$, then $(\sqrt{I})^{n}\subseteq I$.
For $n=2$, Badawi \cite{B} shows that these two conjectures hold. In the case where $R$ is a $u$-ring, we show that these conjectures hold.
In \cite{YB}, Badawi and Yousefian offered a question as follows:\\
$\mathbf{Question.}$
{Let $I$ be a weakly $2$-absorbing ideal of $R$. Is $I$ a strongly weakly $2$-absorbing ideal of $R$?}\\
Regarding this question we will prove that for an arbitrary positive integer $n$, a weakly $n$-absorbing ideal $I$ of a $u$-ring $R$ is a
strongly weakly $n$-absorbing ideal of $R$.

\section{Properties of weakly $n$-absorbing ideals}
Let $n$ be a positive integer. It is obvious that any $n$-absorbing ideal of a ring $R$ is a weakly $n$-absorbing ideal of $R$, also the zero ideal is a weakly $n$-absorbing ideal of $R$, by definition. Therefore $I=\{0\}$ is a weakly $n$-absorbing ideal of the ring $\mathbb{Z}_{2^{n+1}}$, but it is easy to see that $I$ is not an $n$-absorbing ideal of $\mathbb{Z}_{2^{n+1}}$.\\
Consider elements $a_1,\dots,a_n$ and ideals $I_{1},\dots,I_n$ of a ring $R$. Throughout this paper we use the following notations: \\
$a_{1}\cdots \widehat{a_{i}}\cdots a_{n}$: $i$-${th}$ term is excluded from $a_{1}\cdots a_{n}$.\\
Similarly; $I_{1}\cdots \widehat{I_{i}}\cdots I_{n}$: $i$-${th}$ term is excluded from $I_{1}\cdots I_{n}$.\\
Moreover, ${\rm Nil}(R)$ denotes the ideal of nilpotent elements of $R$. 
\begin{theorem}\label{intersection}
Let $R$ be a ring and let $m$ and $n$ be positive integers.
\begin{enumerate}
\item A proper ideal $I$ of $R$ is a weakly $n$-absorbing ideal if and only if whenever $0\neq x_{1}\cdots x_{m}\in I$
for $x_{1},\dots,x_{m}\in R$ with $m>n$, then there are $n$ of the $x_{i}$'s whose product is in $I$.
\item  If $I$ is a weakly $n$-absorbing ideal of $R$, then $I$ is a weakly $m$-absorbing ideal of $R$ for all $m\geq n$.
\item If $I_i$ is a weakly $n_i$-absorbing ideal of $R$ for each $1\leq i\leq k$, then $I_1\cap \cdots \cap I_k$ is a weakly $n$-absorbing ideal of $R$ for $n = n_1+\cdots +n_k$. In particular, if $P_1,\dots,P_n$ are weakly prime ideals of $R$, then  $P_1\cap \cdots \cap P_n$ is a weakly $n$-absorbing ideal of $R$.
\item If $P_1,\dots,P_n$ are weakly prime ideals of $R$ that are pairwise comaximal ideals, then $I = P_1\cdots  P_n$ is a weakly $n$-absorbing ideal of $R$.
\end{enumerate}
\end{theorem}

\begin{proof}
The proof of (1) and (2) are routine, so it is left out.\\
(3) Let $a_{1},\dots,a_{n+1}\in R$ such that
$0\neq a_{1}\cdots a_{n+1}\in I_{1}\cap\cdots\cap I_{k}$.
Since $I_{i}$'s are weakly $n_{i}$-absorbing, then, for each $1\leq i\leq k$, there exist integers $1\leq
j_{1},j_{2},\dots,j_{n_{i}}\leq n+1$ such that
$a_{j_{1}}a_{j_{2}}\cdots a_{j_{n_{i}}}\in I_{i}$. So
we have $a_{1_{1}}a_{1_{2}}\cdots a_{1_{n_{1}}}a_{2_{1}}a_{2_{2}}\cdots
a_{2_{n_{2}}}\cdots a_{n_{1}}a_{n_{2}}\cdots a_{n_{n_{k}}}\in I_{1}\cap\cdots\cap I_{k}$
which implies that $I_{1}\cap\cdots\cap I_{k}$ is weakly $n$-absorbing.\\
(4) is a direct consequence of (3).
\end{proof}

\begin{proposition}
Let $I$ be a proper ideal of a ring $R$. Then the following conditions are equivalent:
\begin{enumerate}
\item $I$ is strongly weakly $n$-absorbing;
\item For any ideals $I_{1},\dots,I_{n+1}$ of $R$ such that
$I\subseteq I_{1}$, $0\neq I_{1}\cdots I_{n+1}\subseteq I$ implies that
there are $n$ of $I_{i}$'s whose product is in $I$.
\end{enumerate}
\end{proposition}

\begin{proof}
$(1)\Rightarrow(2)$ is clear.\\
$(2)\Rightarrow(1)$ Let $J,I_{2},\dots,I_{n+1}$ be ideals of $R$ such that
$0\neq JI_{2}\cdots I_{n+1}\subseteq I$. Then we
have that $0\neq(J+I)I_{2}\cdots I_{n+1}=(JI_{2}\cdots I_{n+1})+(II_{2}\cdots I_{n+1})\subseteq I$.
Set $I_{1}:=J+I$. Then, by hypothesis $I_{2}\cdots I_{n+1}\subseteq I$ or
there exists $2\leq i\leq n+1$ such that $(J+I)I_{2}\cdots\widehat{I_{i}}\cdots I_{n+1}\subseteq I$. Therefore,
$I_{2}\cdots I_{n+1}\subseteq I$ or
there exists $2\leq i\leq n+1$ such that $JI_{2}\cdots\widehat{I_{i}}\cdots I_{n+1}\subseteq I$.
So $I$ is strongly weakly $n$-absorbing.
\end{proof}

\begin{theorem}\label{T7}
Let $R$ be a commutative ring and $J$ be a weakly $n$-absorbing ideal of $R$.
\begin{enumerate}
\item If $I$ is an ideal of $R$ with $I\subseteq J$, then $J/I$ is a weakly $n$-absorbing ideal of $R/I$.
\item If $T$ is a subring of $R$, then $J\cap T$ is a weakly $n$-absorbing ideal of $T$.
\item If $S$ is a multiplicatively closed subset of $R$ with $J\cap S = \emptyset$, then $J_S$ is a weakly $n$-absorbing ideal of $R_S$.
\end{enumerate}
\end{theorem}

\begin{proof}
(1) Let $\bar{R}= R/I,\bar{J} = J/I$ and $\bar{a}_1, \bar{a}_2,\dots,\bar{a}_{n+1}\in \bar{R}$ such that $\bar{a}_1 \bar{a}_2\cdots \newline\bar{a}_{n+1}\in \bar{J}\setminus \{0\}$. Since $\bar{a}_1 \bar{a}_2 \cdots \bar{a}_{n+1}\neq 0$, so $a_1 a_2\cdots a_{n+1} \in R\setminus I$. Hence $a_1 a_2\cdots a_{n+1} \in J\setminus \{0\}$. As $J$ is a weakly $n$-absorbing ideal of $R$, we have $n$ of $a_i$'s whose product is in $J$. Then there are $n$ of $\bar{a}_i$'s whose product is in $\bar{J}$.\\
(2) It's obvious.\\ 
(3) Assume that $0\neq(a_1/s_1)(a_2/s_2)\cdots(a_{n+1}/s_{n+1})\in J_S$  such that $a_1,a_2,...,\\a_{n+1}\in R$ and $s_1,s_2,\dots,s_{n+1}\in S$ and $$(a_1/s_1)(a_2/s_2)\cdots\widehat{(a_i/s_i)}\cdots(a_{n+1}/s_{n+1})\notin J_S,$$ for any $2\leq i\leq n+1$. Now let $(a_1 a_2\cdots a_{n+1})/( s_1 s_2\cdots s_{n+1}) = (x/u)$ for some $x\in J$ and $u\in S$. Then there exists $v\in S$ such that $vu(a_1 a_2\cdots a_{n+1}) = vx(s_1 s_2\cdots s_{n+1})$. So we have $(vua_1)a_2\cdots a_{n+1}\in J\setminus \{0\}$ but the product of $(vua_1)$ with $n-1$ of $a_i$'s for  $2\leq i\leq n+1$ is not in $J$. So we conclude $a_2\cdots a_{n+1}\in J$ and then $(a_2/s_2)\cdots (a_{n+1}/s_{n+1})\in J_S$, that is, $J_S$ is a weakly $n$-absorbing ideal of $R_S$.
\end{proof}

\begin{theorem}\label{factor}
Let $I\subseteq J$ be proper ideals of a ring $R$. If $I$ is a weakly $n$-absorbing ideal of $R$ and $J/I$ is a weakly $n$-absorbing ideal of $R/I$, then $J$ is a  weakly $n$-absorbing ideal of $R$.
\end{theorem}

\begin{proof}
Suppose that $I$ is a weakly $n$-absorbing ideal of $R$ and $J/I$ is a weakly $n$-absorbing ideal of $R/I$. Let $0\neq a_{1}\cdots a_{n+1}\in J$ where $a_{1},\dots,a_{n+1}\in R$, so $(a_{1}+I)\cdots(a_{n+1}+I)\in J/I$. If $a_{1}\cdots a_{n+1}\in I$, then for some $1\leq i\leq n+1$, $a_{1}\cdots\widehat{a_{i}}\cdots a_{n+1}\in I\subseteq J$, because $I$ is weakly $n$-absorbing. If $a_{1}\cdots a_{n+1}\notin I$, then for some $1\leq i\leq n+1$, $(a_{1}+I)\cdots\widehat{(a_{i}+I)}\cdots (a_{n+1}+I)\in J/I$, since $J/I$ is a weakly $n$-absorbing ideal of $R/I$. So $a_{1}\cdots\widehat{a_{i}}\cdots a_{n+1}\in J$. Consequently $J$ is a  weakly $n$-absorbing ideal of $R$.
\end{proof}

\begin{theorem}
Let $I$ be an ideal of an integral domain $R$. Then $\langle I,X\rangle$ is a weakly $n$-absorbing ideal of
$R[X]$ if and only if $I$ is a weakly $n$-absorbing ideal of $R$. 
\end{theorem}
\begin{proof}
By Theorem \ref{T7}(1), Theorem \ref{factor} and regarding the isomorphism $\langle I,X\rangle/\langle X\rangle\simeq I$ in $R[X]/\langle X\rangle\simeq R$ we have the result.
\end{proof}

\begin{proposition}
Let $I$ be a weakly primary ideal of a ring $R$, and let $(\sqrt{I})^{n}\subseteq I$ for some positive integer $n$ (for example, if $\sqrt{I}$ is a finitely generated ideal). Then $I$ is a weakly $n$-absorbing ideal of $R$. 
\end{proposition}
\begin{proof}
Let $0\neq a_{1}\cdots a_{n+1}\in I$ for $a_{1},\dots,a_{n+1}\in R$. If one of the $a_{i}$'s is not in $\sqrt{I}$, then
the product of the other $a_{i}$'s is in $I$, since $I$ is weakly primary. Thus we may assume that
every $a_{i}$ is in $\sqrt{I}$. Since $(\sqrt{I})^{n}\subseteq I$, we have $a_{1}\cdots a_{n}\in I$. Hence $I$ is a weakly $n$-absorbing ideal of $R$.
\end{proof}


\begin{remark}
Let $R$ be a ring such that its zero ideal is $n$-absorbing (e.g., let $R$ be an integral domain). 
Then every weakly $n$-absorbing ideal of $R$ is an $n$-absorbing ideal.
\end{remark}

Let $M$ be an $R$-module. We say that $M$ is secondary precisely when $M\neq0$ and, for each $r\in R$, either $rM=M$ or there
exists $n\in\mathbb{N}$ such that $r^{n}M=0$. When this is the case, $P:=\sqrt{(0:_{R}M)}$
is a prime ideal of $R$: in these circumstances, we say that $M$ is a $P$-secondary
$R$-module. A secondary ideal of $R$ is just a secondary submodule
of the $R$-module $R$ (see \cite{Ma}).

\begin{theorem}
Let $I$ be a secondary ideal of a ring $R$. If $J$ is a weakly $n$-absorbing ideal of $R$, then $I\cap J$ is secondary.
\end{theorem}
\begin{proof}
Let $I$ be a $P$-secondary ideal of $R$, and let $a\in R$. If $a\in P=\sqrt{(0:_{R}I)}$, then clearly $a\in\sqrt{(0:_{R}I\cap J)}$. If $a\notin P$, then $a^{n}\notin P$, and so $a^{n}I=I$. We calim that $a(I\cap J)=I\cap J$. Assume that $0\neq x\in I\cap J$. There is an element $b\in I$ such that $x=a^{n}b\in J$. Since $J$ is  weakly $n$-absorbing we have either $a^{n}\in J$ or $a^{n-1}b\in J$. If $a^{n}\in J$, then $I=a^{n}I\subseteq J$ and so $a(I\cap J)=aI=I=I\cap J$. If $a^{n-1}b\in J$, then $x=a^{n}b\in a(I\cap J)$ and we are done.
\end{proof}

A weakly prime ideal $P$ of a ring $R$ is said to be a divided weakly prime ideal if $P\subset xR$ for every $x \in R\backslash P$; thus a divided weakly prime ideal is comparable to every ideal of $R$.
\begin{theorem} 
Let $P$ be a divided weakly prime ideal of a ring $R$, and let $I$ be a weakly $n$-absorbing
ideal of $R$ with $\sqrt{I}=P$. Then $I$ is a weakly primary ideal of R.
\end{theorem}
\begin{proof}
Let $0\neq xy\in I$ for $x,y\in R$ and $y\notin P$. Then $x\in P$. If $y^{n-1}=0$, then $y\in\sqrt{I}=P$, which is a contradiction. Therefore $y^{n-1}\neq0$, and so $y^{n-1}\notin P$. Thus $P\subset y^{n-1}R$, because $P$ is a divided weakly prime ideal of $R$. Hence $x=y^{n-1}z$ for some $z\in R$. As $0\neq y^{n}z=yx\in I$, $y^{n}\notin I$, and $I$ is a weakly $n$-absorbing ideal of $R$, we have $x=y^{n-1}z\in I$. Hence $I$ is a weakly primary ideal of $R$. 
\end{proof}

Let $I$ be a weakly $n$-absorbing ideal of a ring $R$ and $a_{1},..., a_{n+1}\in R$. We say $(a_{1},\dots,a_{n+1})$
is an $(n+1)$-tuple-zero of $I$ if $a_{1}\cdots a_{n+1}=0$, and for each $1\leq i\leq n+1$, $a_{1}\cdots\widehat{a_{i}}\cdots a_{n+1}\notin I$.

In the following Theorem $a_{1}\cdots\widehat{a_{i}}\cdots\widehat{a_{j}}\cdots a_{n}$ denotes that $a_{i}$ and $a_{j}$ are eliminated from $a_{1}\cdots a_{n}$.

\begin{theorem}\label{first}
Let $I$ be a weakly $n$-absorbing ideal of a ring $R$ and suppose that
$(a_{1},\dots,a_{n+1})$ is an $(n+1)$-tuple-zero of $I$ for some $a_{1},..., a_{n+1}\in R$. Then for every $1\leq \alpha_{1},\alpha_{2},\dots,\alpha_{m}\leq n+1$ which $1\leq m\leq n$, $$a_{1}\cdots\widehat{a_{\alpha_{1}}}\cdots\widehat{a_{\alpha_{2}}}\cdots\widehat{a_{\alpha_{m}}}\cdots a_{n+1}I^{m}=\{0\}.$$
\end{theorem}

\begin{proof}
We use induction on $m$. Let $m=1$ and suppose that $a_{1}\cdots\widehat{a_{\alpha_{1}}}\cdots a_{n+1}\newline x\neq0$ for some $x\in I$. Then $a_{1}\cdots\widehat{a_{\alpha_{1}}}\cdots a_{n+1}(a_{\alpha_{1}}+x)\neq0$. Since $I$ is weakly $n$-absorbing and
$a_{1}\cdots\widehat{a_{\alpha_{1}}}\cdots a_{n+1}\notin I$, we conclude that $a_{1}\cdots\widehat{a_{\alpha_{1}}}\cdots\widehat{a_{\alpha_{2}}}\cdots\newline a_{n+1}(a_{\alpha_{1}}+x)\in I$, for some $1\leq\alpha_{2}\leq n+1$ distinct from $\alpha_{1}$. Hence $a_{1}\cdots\widehat{a_{\alpha_{2}}}\cdots a_{n+1}\in I$, a contradiction. Thus $a_{1}\cdots\widehat{a_{\alpha_{1}}}\cdots a_{n+1}I=\{0\}$.\\
Now suppose $m>1$ and assume that for all integers less than $m$ the claim holds.
Let $a_{1}\cdots\widehat{a_{\alpha_{1}}}\cdots\widehat{a_{\alpha_{2}}}\cdots\widehat{a_{\alpha_{m}}}\cdots a_{n+1}x_{1}x_{2}\cdots x_{m}\neq0$ for some $x_{1},x_{2},\dots,x_{m}\in I$. 
By  induction hypothesis, we conclude that

{\small $$
\begin{array}{ll}
a_{1}\cdots\widehat{a_{\alpha_{1}}}\cdots &\widehat{a_{\alpha_{2}}}\cdots\widehat{a_{\alpha_{m}}}\cdots
a_{n+1}(a_{\alpha_{1}}+x_{1})(a_{\alpha_{2}}+x_{2})\cdots(a_{\alpha_{m}}+x_{m})\\
&=a_{1}\cdots\widehat{a_{\alpha_{1}}}\cdots\widehat{a_{\alpha_{2}}}\cdots\widehat{a_{\alpha_{m}}}\cdots a_{n+1}x_{1}x_{2}\cdots x_{m}\neq0.
\end{array}
$$}
 Hence either
 \begin{eqnarray*}
a_{1}\cdots\widehat{a_{\alpha_{1}}}\cdots\widehat{a_{\alpha_{2}}}\cdots\widehat{a_{\alpha_{m}}}\cdots
a_{n+1}(a_{\alpha_{1}}+x_{1})\cdots\widehat{(a_{\alpha_{i}}+x_{i})}\cdots(a_{\alpha_{m}}+x_{m})\in I
 \end{eqnarray*}
 for some $1\leq i\leq m$; or
 \begin{eqnarray*}
a_{1}\cdots\widehat{a_{\alpha_{1}}}\cdots\widehat{a_{\alpha_{2}}}\cdots\widehat{a_{\alpha_{m}}}\cdots\widehat{a_{j}}\cdots
a_{n+1}(a_{\alpha_{1}}+x_{1})(a_{\alpha_{2}}+x_{2})\cdots(a_{\alpha_{m}}+x_{m})\in I,
 \end{eqnarray*}
for some $j$ distinct from ${\alpha_{i}}$'s. Thus either $a_{1}\cdots{a_{\alpha_{1}}}\cdots\widehat{a_{\alpha_{i}}}\cdots{a_{\alpha_{m}}}\cdots a_{n+1}\in I$ or $a_{1}\cdots{a_{\alpha_{1}}}\cdots{a_{\alpha_{m}}}\cdots\widehat{a_{j}}\cdots a_{n+1}\in I$,
a contradiction. Thus $$a_{1}\cdots\widehat{a_{\alpha_{1}}}\cdots\widehat{a_{\alpha_{2}}}\cdots\widehat{a_{\alpha_{m}}}\cdots a_{n+1}I^{m}=\{0\}.$$
\end{proof}

Now we state a version of Nakayama's lemma.
\begin{theorem}\label{nil}
Let $I$ be a weakly $n$-absorbing ideal of $R$ that is not an $n$-absorbing
ideal. Then
\begin{enumerate}
\item $I^{n+1}=\{0\}$.
\item $\sqrt{I}=Nil(R)$.
\item If $M$ is an $R$-module and $IM=M$, then $M=\{0\}$.
\end{enumerate}
\end{theorem}

\begin{proof}
(1) Since $I$ is not an $n$-absorbing ideal of $R$, $I$ has an $(n+1)$-truple-zero $(a_{1},\dots,a_{n+1})$ for some
$a_{1},\dots,a_{n+1}\in R$. Suppose that $x_{1}x_{2}\cdots x_{n+1}\neq0$ for some $x_{1},x_{2},\dots,x_{n+1}\in I$. Then by Theorem \ref{first} we have $$(a_{1}+x_{1})\cdots(a_{n+1}+x_{n+1})=x_{1}x_{2}\cdots x_{n+1}\neq0.$$ Hence $(a_{1}+x_{1})\cdots\widehat{(a_{i}+x_{i})}\cdots(a_{n+1}+x_{n+1})\in I$ for some $1\leq i\leq n+1$. Thus $a_{1}\cdots\widehat{a_{i}}\cdots a_{n+1}\in I$, a contradiction. Hence  $I^{n+1}=\{0\}$.\\
(2) Clearly, ${\rm Nil}(R)\subseteq\sqrt{I}$. As $I^{n+1}=\{0\}$, we get $I\subseteq{\rm Nil}(R)=\sqrt{\{0\}}$; hence $\sqrt{I}\subseteq{\rm Nil}(R)$, as required.\\
(3) Since $IM=M$, we have $M=IM=I^{n+1}M=\{0\}$.
\end{proof}

The following example shows that a proper ideal $I$ of a ring $R$ with $I^{n+1}=\{0\}$ need not be a weakly
$n$-absorbing ideal of $R$.
\begin{example}
Let $R=\mathbb{Z}_{2^{n+2}}$. Then $I=\{0,2^{n+1}\}$ is an ideal of $\mathbb{Z}_{2^{n+2}}$ and $I^{n+1}=\{0\}$, but
$2\cdots2=2^{n+1}\in I$ and $2^{n}\notin I$.
\end{example}


\begin{corollary}
Let $R$ be a ring such that ${Nil}(R)$ is an $n$-absorbing (resp. a weakly $n$-absorbing) ideal of $R$. If $I$ is a weakly $n$-absorbing ideal of $R$, then $\sqrt{I}$ is an $n$-absorbing (resp. a weakly $n$-absorbing) ideal of $R$.
\end{corollary}
\begin{proof}
Assume that ${\rm Nil}(R)$ is an $n$-absorbing (resp. a weakly $n$-absorbing) ideal of $R$ and $I$ is a weakly $n$-absorbing ideal of $R$. If $I$ is an $n$-absorbing ideal of $R$, then $\sqrt{I}$ is an $n$-absorbing ideal, \cite[Theorem 2.1(e)]{AB1} and so $\sqrt{I}$ is a weakly $n$-absorbing ideal. If $I$ is not an $n$-absorbing ideal of $R$, 
then by Theorem \ref{nil} and by our hypothesis, $\sqrt{I}={\rm Nil}(R)$ which is an $n$-absorbing (resp. a weakly $n$-absorbing)  ideal.
\end{proof}

\begin{theorem}
Let $I$ be a weakly $n$-absorbing ideal of a ring $R$ that is not $n$-absorbing and let $J$ be a weakly $m$-absorbing ideal of $R$ that is not $m$-absorbing, and $n\geq m$. Then $I+J$ is a weakly $n$-absorbing ideal of $R$. In particular, $\sqrt{I+J}=Nil(R)$.
\end{theorem}
\begin{proof}
By Theorem \ref{nil}, we have $\sqrt{I}+\sqrt{J}=\sqrt{0}\neq R$, so $I+J$
is a proper ideal of $R$. Since $(I+J)/J\simeq I/(I\cap J)$ and $I$ is weakly $n$-absorbing, we get that $(I+J)/J$ is a weakly $n$-absorbing ideal of $R/J$, by Theorem \ref{T7}(1). On the other hand $J$ is also weakly $n$-absorbing, by Theorem \ref{intersection}(2). Now, the assertion follows from Theorem \ref{factor}. Finally, by \cite[2.25(i)]{Sh} we have $\sqrt{I+J}=\sqrt{\sqrt{I}+\sqrt{J}}=\sqrt{\sqrt{0}}={\rm Nil}(R)$.
\end{proof}

Let $R$ be a ring and $M$ an $R$-module. A submodule $N$ of $M$ is called a pure submodule
if the sequence $0\rightarrow N\otimes_{R} E\rightarrow M\otimes_{R} E$ is exact for every $R$-module $E$.

As another consequence of Theorem \ref{nil} we have the following corollary.
\begin{corollary}
Let $R$ be a ring. Then the following conditions hold:
\begin{enumerate}
\item Every nonzero weakly $n$-absorbing ideal of $R/Nil(R)$ is $n$-absorbing.
\item If $I$ is a pure weakly $n$-absorbing ideal of R that is not $n$-absorbing, then $I=\{0\}$.
\item If $R$ is von Neumann regular ring, then the only weakly $n$-absorbing ideal of $R$ that is not
$n$-absorbing can only be $\{0\}$.
\end{enumerate}
\end{corollary}
\begin{proof}
(1) Notice that ${\rm Nil}(\frac{R}{{\rm Nil}(R)})=\{0\}$.\\
(2), (3) Note that every pure ideal is idempotent, and every ideal of a von Neumann regular ring is pure (see \cite{F}).
\end{proof}

\begin{theorem}
Let $I$ be a weakly $n$-absorbing ideal of $R$ that is not an $n$-absorbing
ideal. Then
\begin{enumerate}
\item  If $w\in Nil(R)$, then either $w^{n}\in I$ or $w^{n-i}I^{i+1}=\{0\}$ for every $0\leq i\leq n-1$.
\item ${\mathcal N}^{n}I^{n}=\{0\}$, in which  ${\mathcal N}$ denotes the ideal of $R$ generated by all the $(n-1)$-th powers of elements of $Nil(R)$.
\end{enumerate}
\end{theorem}
\begin{proof}
(1) Suppose that $w\in{\rm Nil}(R)$ and $w^{n}\notin I$. We show that for every $0\leq i\leq n-1$, $w^{n-i}I^{i+1}=\{0\}$. We use induction on $i$. For the first step, fix $i=0$. Assume $w^{n}I\neq0$. Let $m$ be the least positive integer such that $w^{m}=0$.
Then $m\geq n+1$ ($w^{n}\notin I$) and for some $x\in I$ we have $w^{n}(x+w^{m-n})=w^{n}x\neq0$. Since $w^{n}\notin I$ and $I$ is a weakly $n$-absorbing ideal, then $(w^{n-1}x+w^{m-1})\in I$. Hence $w^{m-1}\in I$. On the other hand $w^{m-1}\neq 0$. Therefore  $w^{n}\in I$, which is a contradiction. Now, assume that for every $0\leq j< i$ the claim holds. We will show that $w^{n-i}I^{i+1}=\{0\}$. 
Assume that $w^{n-i}x_{1}x_{2}\cdots x_{i+1}\neq0$. By hypothesis we have 
\[w^{n-i}(w+x_{1})\cdots(w+x_{i})(w^{m-n}+x_{i+1})=w^{m}+w^{m-1}(\sum_{1\leq r\leq i} x_{r})+\]
\[w^{m-2}(\sum_{\substack{1\leq r,s\leq i\\\ r\neq s}} x_{r}x_{s})+\cdots+w^{m-i}x_{1}\cdots x_{i}+w^{n}x_{i+1}+w^{n-1}(\sum_{1\leq r\leq i} x_{r})x_{i+1}\]
\[+w^{n-2}(\sum_{\substack{1\leq r,s\leq i\\\ r\neq s}} x_{r}x_{s}) x_{i+1}+\cdots+w^{n-i}x_{1}\cdots x_{i}x_{i+1}=w^{n-i}x_{1}\cdots x_{i+1}.\]

Therefore, either $w^{n-i}(w+x_{1})\cdots(w+x_{i})\in I$ or for some $1\leq t\leq i$, $w^{n-i}(w+x_{1})\cdots\widehat{(w+x_{t})}\cdots(w+x_{i})(w^{m-n}+x_{i+1})\in I$ or $w^{n-i-1}(w+x_{1})\cdots(w+x_{i})(w^{m-n}+x_{i+1})\in I$ which the first case implies that $w^{n}\in I$, a contradiction, and two other cases imply that $w^{m-1}\in I$. Now $w^{m-1}\neq0$ again shows that $w^{n}\in I$, a contradiction.\\
(2) Let $a_{1},\dots,a_{n}\in{\rm Nil}(R)$. If at least one of the $a_{i}^{n}$'s does not belong to $I$, then $a_{1}\cdots a_{n}I^{n}=0$, by part  (1). Therefore, $a_{1}^{n-1}\cdots a_{n}^{n-1}I^{n}=0$. Hence suppose that for every $1\leq i\leq n$, $a_{i}^{n}\in I$. Then $a_{1}\cdots a_{n}(a^{n-1}_{1}+\cdots+a_{n}^{n-1})\in I$. If
$(a_{1},\dots,a_{n},a^{n-1}_{1}+\cdots+a_{n}^{n-1})$ is an $(n+1)$-tuple-zero of $I$, then $a_{1}\cdots a_{n}I=0$, by Theorem \ref{first}, and hence $a_{1}^{n-1}\cdots a_{n}^{n-1}I^{n}=0$. If $(a_{1},\dots,a_{n},a^{n-1}_{1}+\cdots+a_{n}^{n-1})$ is not an $(n+1)$-tuple-zero of $I$, then we can easily see that there is an $1\leq i\leq n$ such that $a_{1}\cdots a_{i}^{n-1}\cdots a_{n}\in I$ or $a_{1}\cdots a_{n}\in I$. Hence $a_{1}^{n-1}\cdots a_{n}^{n-1}\in I$, and so $a_{1}^{n-1}\cdots a_{n}^{n-1}I^{n}=0$, by Theorem \ref{nil}(1). Consequently ${\mathcal N}^{n}I^{n}=\{0\}$.
\end{proof}

\begin{theorem}\label{T4}
Let $R= R_1\times\cdots\times R_s$ be a decomposable commutative ring and let $L=I_{1}\times\dots\times I_{\alpha_{1}-1}\times R_{\alpha_{1}}\times I_{\alpha_{1}+1}\times\dots\times I_{\alpha_{j}-1}\times R_{\alpha_{j}}\times I_{\alpha_{j}+1}\times\cdots\times I_{s}$ be an ideal of $R$ in which $\{\alpha_{1},\dots,\alpha_{j}\}\subset\{1,\dots,s\}$. The following conditions are equivalent:
\begin{enumerate}
\item $L$ is a weakly $n$-absorbing ideal of $R$;
\item $L$ is an $n$-absorbing ideal of $R$;
\item $L':=I_{1}\times\dots\times I_{\alpha_{1}-1}\times I_{\alpha_{1}+1}\times\dots\times I_{\alpha_{j}-1}\times I_{\alpha_{j}+1}\times\cdots\times I_{s}$ is an $n$-absorbing ideal of $R':=R_{1}\times\dots\times R_{\alpha_{1}-1}\times R_{\alpha_{1}+1}\times\dots\times R_{\alpha_{j}-1}\times R_{\alpha_{j}+1}\times\cdots\times R_{s}$.
\end{enumerate}

\end{theorem}

\begin{proof}
$(1)\Rightarrow (2)$ Clearly $L\nsubseteq Nil(R)$, so by Theorem \ref{nil}(2), $L$ is an $n$-absorbing ideal of $R$.\\ 
$(2)\Rightarrow (3)$ Assume that $L$ is an $n$-absorbing ideal of $R$ and

{\small $$
\begin{array}{ll}
(a_{1}^{(1)},\dots,a_{\alpha_{1}-1}^{(1)},&a_{\alpha_{1}+1}^{(1)},\dots,a_{\alpha_{j}-1}^{(1)},
a_{\alpha_{j}+1}^{(1)},\dots,a_{s}^{(1)})\cdots\\
&(a_{1}^{(n+1)},\dots,a_{\alpha_{1}-1}^{(n+1)},
a_{\alpha_{1}+1}^{(n+1)},\dots,a_{\alpha_{j}-1}^{(n+1)},a_{\alpha_{j}+1}^{(n+1)},\dots,a_{s}^{(n+1)})\in L',
\end{array}
$$}
in which for every $1\leq t\leq n+1$, $a_{i}^{(t)}$'s are in $R_{i}$, respectively. Then 

{\small $$
\begin{array}{ll}
(a_{1}^{(1)},\dots,a_{\alpha_{1}-1}^{(1)},&1,a_{\alpha_{1}+1}^{(1)},\dots,a_{\alpha_{j}-1}^{(1)},1,
a_{\alpha_{j}+1}^{(1)},\dots,a_{s}^{(1)})\cdots\\
&(a_{1}^{(n+1)},\dots,a_{\alpha_{1}-1}^{(n+1)},1,
a_{\alpha_{1}+1}^{(n+1)},\dots,a_{\alpha_{j}-1}^{(n+1)},1,a_{\alpha_{j}+1}^{(n+1)},\dots,a_{s}^{(n+1)})\in L.
\end{array}
$$}
So there are $n$ of $(a_{1}^{(t)},\dots,a_{\alpha_{1}-1}^{(t)},1,a_{\alpha_{1}+1}^{(t)},\dots,a_{\alpha_{j}-1}^{(t)},1,
a_{\alpha_{j}+1}^{(t)},\dots,a_{n}^{(t)})$'s
whose product is in $L$, because $L$ is an $n$-absorbing ideal of $R$. Thus the product of $n$ of 
$(a_{1}^{(t)},\dots,a_{\alpha_{1}-1}^{(t)},a_{\alpha_{1}+1}^{(t)},\dots,a_{\alpha_{j}-1}^{(t)},
a_{\alpha_{j}+1}^{(t)},\dots,a_{n}^{(t)})$'s
is in $L'$, and so $L'$ is an $n$-absorbing ideal of $R'$.\\
$(3)\Rightarrow (1)$ Let $L'$ is an $n$-absorbing ideal of $R'$. It is routine to see that $L$ is an $n$-absorbing ideal of $R$. Consequently, $L$ is a weakly $n$-absorbing ideal of $R$.
\end{proof}

\begin{theorem}\label{prod}
Let $R=R_{1}\times\cdots\times R_{n}$ where $R_{1},\dots,R_{n}$ are commutative rings with
identity. Suppose that $I_{1}\times I_{2}\times\cdots\times I_{n}$ is an ideal of $R$ which $I_{1}\neq0$ and for each $1\leq i\leq n-1$, $I_{i}$ is a proper ideal of $R_{i}$, and for some $2\leq i\leq n$, $I_{i}$ is a nonzero ideal of $R_{i}$.
The following conditions are equivalent:
\begin{enumerate}
\item $I_{1}\times I_{2}\times\cdots\times I_{n}$ is a weakly $n$-absorbing ideal of $R$;
\item $I_{n}=R_{n}$ and $I_{1}\times I_{2}\times\cdots\times I_{n-1}$ is an $n$-absorbing ideal of $R_{1}\times\cdots\times R_{n-1}$ or $I_{n}$ is a prime ideal of $R_{n}$ and for each $1\leq i\leq n-1$, $I_{i}$ is a prime ideal of $R_{i}$, respectively;
\item $I_{1}\times I_{2}\times\cdots\times I_{n}$ is an $n$-absorbing ideal of $R$.
\end{enumerate}
\end{theorem}
\begin{proof}
(1)$\Rightarrow$(2) Suppose that $I_{1}\times I_{2}\times\cdots\times I_{n}$ is a weakly $n$-absorbing ideal of $R$. If $I_{n}=R_{n}$, then $I_{1}\times I_{2}\times\cdots\times I_{n-1}$ is an $n$-absorbing ideal of $R_{1}\times\cdots\times R_{n-1}$, by Theorem \ref{T4}. Assume that $I_{n}\neq R_{n}$. Fix $2\leq i\leq n$. We show that $I_{i}$ is a prime ideal of $R_i$. Suppose that $ab\in I_{i}$ for some $a,b\in R_{i}$. Let $0\neq x\in I_1$. Then
\begin{eqnarray*}
(x,1,\dots,1)(1,\dots,1,\overbrace{a}^{i-th},1,\dots,1)(1,\dots,1,\overbrace{b}^{i-th},1,\dots,1)(1,0,1,\dots,1,\dots,1)\\
(1,1,0,1,\dots,1,\dots,1)\cdots(1,\dots,1,0,\overbrace{1}^{i-th},\dots,1)(1,\dots,\overbrace{1}^{i-th},0,1,\dots,1)\cdots\\
(1,\dots,1,0)=(x,0,\dots,0,\overbrace{ab}^{i-th},0,\dots,0)\in I_{1}\times\dots\times I_{n}\backslash\{(0,\dots,0)\},
\end{eqnarray*}
Since $I_{1}\times I_{2}\times\cdots\times I_{n}$ is weakly $n$-absorbing and $I_{i}$'s are proper, then either
\begin{eqnarray*}
(x,1,\dots,1)(1,\dots,1,\overbrace{a}^{i-th},1,\dots,1)(1,0,1,\dots,1,\dots,1)(1,1,0,1,\dots,1,\dots,1)\\
\cdots(1,\dots,1,0,\overbrace{1}^{i-th},\dots,1)(1,\dots,\overbrace{1}^{i-th},0,1,\dots,1)\cdots(1,\dots,1,0)\\
=(x,0,\dots,0,\overbrace{a}^{i-th},0,\dots,0)\in I_{1}\times\dots\times I_{n},
\end{eqnarray*}
or
\begin{eqnarray*}
(x,1,\dots,1)(1,\dots,1,\overbrace{b}^{i-th},1,\dots,1)(1,0,1,\dots,1,\dots,1)(1,1,0,1,\dots,1,\dots,1)\\
\cdots(1,\dots,1,0,\overbrace{1}^{i-th},\dots,1)(1,\dots,\overbrace{1}^{i-th},0,1,\dots,1)\cdots(1,\dots,1,0)\\
=(x,0,\dots,0,\overbrace{b}^{i-th},0,\dots,0)\in I_{1}\times\dots\times I_{n},
\end{eqnarray*}
and thus either $a\in I_i$ or $b\in I_i$. Consequently $I_i$ is a prime ideal of $R_i$. Since for some $2\leq i\leq n$, $I_{i}$ is a nonzero ideal of $R_{i}$, similarly we can show that $I_{1}$ is a prime ideal of $R_{1}$.\\
(2)$\Rightarrow$(3) If $I_{n}=R_{n}$ and $I_{1}\times I_{2}\times\cdots\times I_{n-1}$ is an $n$-absorbing ideal of $R_{1}\times\cdots\times R_{n-1}$, then $I_{1}\times I_{2}\times\cdots\times I_{n}$ is an $n$-absorbing ideal of $R$, by Theorem \ref{T4}. Now, 
assume that $I_{n}$ is a prime ideal of $R_{n}$ and for each $1\leq i\leq n-1$, $I_{i}$ is a prime ideal of $R_{i}$. Suppose that 
\[(a_{1}^{(1)},\dots,a_{n}^{(1)})(a_{1}^{(2)},\dots,a_{n}^{(2)})\cdots(a_{1}^{(n+1)},\dots,a_{n}^{(n+1)})\in I_{1}\times I_{2}\times\cdots\times I_{n},\]  
in which $a_{i}^{(j)}$'s are in $R_{i}$. Then for any $1\leq i\leq n$ at least one of the $a_{i}^{(j)}$'s is in $I_{i}$, say $a_{i}^{(i)}$. Thus $(a_{1}^{(1)},\dots,a_{n}^{(1)})(a_{1}^{(2)},\dots,a_{n}^{(2)})\cdots(a_{1}^{(n)},\dots,a_{n}^{(n)})\in I_{1}\times I_{2}\times\cdots\times I_{n}$.
Consequently $I_{1}\times I_{2}\times\cdots\times I_{n}$ is an $n$-absorbing ideal of $R$. \\
(3)$\Rightarrow$(1) is obvious.
\end{proof}

\begin{theorem}
Let $R=R_{1}\times\cdots\times R_{n}$ be a commutative ring, and let for every $1\leq i\leq n-1$, $I_{i}$ be a proper ideal of $R_{i}$ such that $I_{1}\neq0$ and $I_{n}$ be an ideal of $R_{n}$. The following conditions are equivalent:
\begin{enumerate}
\item $I_{1}\times\dots\times I_{n}$ is a weakly $n$-absorbing ideal of $R$ that is not an $n$-absorbing ideal of $R$.
\item $I_{1}$ is a weakly prime ideal of $R_{1}$ that is not a prime ideal and for every $2\leq i\leq n$, $I_{i}=\{0\}$ is a prime ideal of $R_{i}$, respectively.
\end{enumerate}
\end{theorem}
\begin{proof}
(1)$\Rightarrow$(2) Assume that $I_{1}\times\dots\times I_{n}$ is a weakly $n$-absorbing ideal of $R$ that is not an $n$-absorbing ideal. If for some $2\leq i\leq n$ we have $I_i\neq\{0\}$, then $I_{1}\times\dots\times I_{n}$ is an $n$-absorbing ideal of $R$ by Theorem \ref{prod}, which contradicts our assumption. Thus for every $2\leq i\leq n$, $I_i=\{0\}$. A proof similar to part (1)$\Rightarrow$(2) of Theorem \ref{prod} shows that for every $2\leq i\leq n$, $I_{i}=\{0\}$ is a prime ideal of $R_i$. Now, we show that $I_1$ is a weakly prime ideal of $R_1$. Consider $a,b\in R_1$ such that $0\neq ab\in I_1$. Note that 
\begin{eqnarray*}
(a,1,\dots,1)(b,1,\dots,1)(1,0,1...,1)
(1,1,0,1,\dots,1)\cdots(1,\dots,1,0)\\=(ab,0,\dots,0)\in (I_{1}\times\{0\}\times\cdots\times\{0\})\backslash\{(0,\dots,0)\}.
\end{eqnarray*}
Since $I_{1}\times\{0\}\times\cdots\times\{0\}$ is a weakly $n$-absorbing ideal of $R$, we have either $(a,0,\dots,0)\in I_{1}\times\{0\}\times\cdots\times\{0\}$ or $(b,0,\dots,0)\in I_{1}\times\{0\}\times\cdots\times\{0\}$. So either $a\in I_1$ or $b\in I_1$. Thus $I_1$ is a weakly prime ideal of $R_1$. Assume $I_1$ is a 
prime ideal of $R_1$, since for every $2\leq i\leq n$, $I_i$ is a 
prime ideal of $R_i$, it is easy to see that $I_{1}\times\dots\times I_{n}$ is an
 $n$-absorbing ideal of $R$, which is a contradiction.\\
(2)$\Rightarrow$(1)
Suppose that $I_{1}$ is a weakly prime ideal of $R_{1}$ that is not a prime ideal and for every $2\leq i\leq n$, $I_{i}=\{0\}$ is a prime ideal of $R_{i}$. Assume that 
\begin{eqnarray*}
(a_{1}^{(1)},\dots,a_{n}^{(1)})(a_{1}^{(2)},\dots,a_{n}^{(2)})\cdots(a_{1}^{(n+1)},\dots,a_{n}^{(n+1)})\\
\in I_{1}\times\{0\} \times\cdots\times\{0\}\backslash\{(0,\dots,0)\} 
\end{eqnarray*}
in which $a_{i}^{(j)}$'s are in $R_{i}$. Then at least one of the $a_{1}^{(j)}$'s is in $I_{1}$, say $a_{1}^{(1)}$, and for any $2\leq i\leq n$ at least one of the $a_{i}^{(j)}$'s is zero, say $a_{i}^{(i)}$. Thus $(a_{1}^{(1)},\dots,a_{n}^{(1)})(a_{1}^{(2)},\dots,a_{n}^{(2)})\cdots(a_{1}^{(n)},\dots,a_{n}^{(n)})\in I_{1}\times\{0\} \times\cdots\times\{0\}$.
Consequently $I_{1}\times\{0\} \times\cdots\times\{0\}$ is an weakly $n$-absorbing ideal of $R$. Since $I_1$ is not a prime ideal of $R_{1}$, there exist elements $a,b\in R_{1}$ such that $ab=0$,
but $a\notin I_1$ and $b\notin I_1$. Hence 
{\small$$(a,1,\dots,1)(b,1,\dots,1)(1,0,1,\dots,1)(1,1,0,1,\dots,1)\cdots(1,\dots,1,0)=(0,\dots,0),$$} 
but neither $$(a,1,\dots,1)(1,0,1,\dots,1)(1,1,0,1,\dots,1)\cdots(1,\dots,1,0)\in I_{1}\times\{0\} \times\cdots\times\{0\},$$ nor $$(b,1,\dots,1)(1,0,1,\dots,1)(1,1,0,1,\dots,1)\cdots(1,\dots,1,0)\in I_{1}\times\{0\} \times\cdots\times\{0\},$$ also the product of $(a,1,\dots,1)(b,1,\dots,1)$ with any $n-2$ of elements \newline$(1,0,1,\dots,1),(1,1,0,1,\dots,1),\dots,(1,\dots,1,0)$ is not in $I_{1}\times\{0\} \times\cdots\times\{0\}$. Consequently $I_{1}\times\{0\} \times\cdots\times\{0\}$ is not an $n$-absorbing ideal of $R$.
\end{proof}

\begin{theorem}\label{nonzero}
Let $R=R_{1}\times\cdots\times R_{n+1}$ where $R_{i}$'s are commutative
rings with identity. If $I$ is a weakly $n$-absorbing ideal of $R$, then either $I=\{(0,\dots,0)\}$, or $I$ is an $n$-absorbing ideal of $R$.
\end{theorem}
\begin{proof}
We know that the ideal $I$ is of the form $I_{1}\times\cdots\times I_{n+1}$ where $I_{i}$'s are ideals of
$R_{i}$'s, respectively. Since $\{(0,\dots,0)\}$ is a weakly $n$-absorbing ideal of $R$, we may assume
that $I=I_{1}\times\cdots\times I_{n+1}\neq\{(0,\dots,0)\}$. So, there is an element
$(0,\dots,0)\neq(a_{1},\dots,a_{n+1})\in I$. Then $$(a_{1},1,\dots,1)(1,a_{2},1,\dots,1)\cdots(1,\dots,1,a_{n+1})\in I.$$ Since $I$ is a weakly $n$-absorbing ideal of $R$, for some $1\leq i\leq n+1$
\begin{eqnarray*}
(a_{1},1,\dots,1)\cdots(1,...1,a_{i-1},1,\dots,1)(1,...1,a_{i+1},1,\dots,1)\cdots(1,\dots,1,a_{n+1})\\
=(a_{1},\dots,a_{i-1},1,a_{i+1},\dots,a_{n+1})\in I.
\end{eqnarray*}
Then $I_{i}=R_{i}$, for some $1\leq i\leq n+1$. Hence $I\nsubseteq{\rm Nil}(R)$. Therefore, by Theorem \ref{nil}, 
$I$ must be an $n$-absorbing ideal of $R$.
\end{proof}

\begin{theorem}
Let $R=R_{1}\times\cdots\times R_{n+1}$ where $R_{i}$'s are commutative
rings with identity. Let $L=I_{1}\times\dots\times I_{n+1}$ be a nonzero proper ideal of $R$. The following conditions are equivalent:
\begin{enumerate}
\item $L=I_{1}\times\dots\times I_{n+1}$ is a weakly $n$-absorbing ideal of $R$;
\item $L=I_{1}\times\dots\times I_{n+1}$ is an $n$-absorbing ideal of $R$;
\item $L=I_{1}\times\dots\times I_{i-1}\times R_{i}\times I_{i+1}\times\cdots\times I_{n+1}$ for some $1\leq i\leq n+1$ such that for each $1\leq t\leq n+1$ different from $i$, $I_{t}$ is a prime ideal of $R_{t}$ or $L=I_{1}\times\dots\times I_{\alpha_{1}-1}\times R_{\alpha_{1}}\times I_{\alpha_{1}+1}\times\dots\times I_{\alpha_{j}-1}\times R_{\alpha_{j}}\times I_{\alpha_{j}+1}\cdots\times I_{n+1}$ in which $\{\alpha_{1},\dots,\alpha_{j}\}\subsetneq\{1,\dots,n+1\}$ and $$\hspace{-6mm}I_{1}\times\dots\times I_{\alpha_{1}-1}\times I_{\alpha_{1}+1}\times\dots\times I_{\alpha_{j}-1}\times I_{\alpha_{j}+1}\cdots\times I_{n+1}$$ is an $n$-absorbing ideal of $$\hspace{7mm}R_{1}\times\dots\times R_{\alpha_{1}-1}\times R_{\alpha_{1}+1}\times\dots\times R_{\alpha_{j}-1}\times R_{\alpha_{j}+1}\times\cdots\times R_{n+1}.$$
\end{enumerate}
\end{theorem}
\begin{proof}
(1)$\Rightarrow$(2) Since $L$ is a nonzero weakly $n$-absorbing ideal, $L$ is an $n$-absorbing
ideal of $R$ by Theorem \ref{nonzero}.\\ 
(2)$\Rightarrow$(3) Suppose that $L$ is an $n$-absorbing ideal
of $R$, then for some $1\leq i\leq n+1$, $I_{i}=R_{i}$ by the proof
of Theorem \ref{nonzero}. Assume that $L=I_{1}\times\dots\times I_{i-1}\times R_{i}\times I_{i+1}\times\cdots\times I_{n+1}$ for an $1\leq i\leq n+1$ such that for each $1\leq t\leq n+1$ different from $i$, $I_{t}$ is a proper ideal of $R_{t}$. Fix an $I_{t}$ different from $I_{i}$ with $t>i$. Let $ab\in I_{t}$ for some $a,b\in R_{t}$. In this case
{\small $$
\begin{array}{ll}
&(1,\dots,1,\overbrace{a}^{t-th},1,\dots,1)(1,\dots,1,\overbrace{b}^{t-th},1,\dots,1)(0,1,\dots,1)(1,0,1,\dots,1)
\cdots\\
&(1,\dots,1,0,\overbrace{1}^{i-th},\dots,1)(1,\dots,\overbrace{1}^{i-th},0,1,\dots,1)\cdots(1,\dots,1,0,\overbrace{1}^{t-th},\dots,1)\\
&(1,\dots,\overbrace{1}^{t-th},0,1,\dots,1)\cdots(1,\dots,1,0)=(0,\dots,0,\overbrace{1}^{i-th},0,\dots,0,\overbrace{ab}^{t-th},0,\dots,0)\in L.
\end{array}
$$}
Since $I_{1}\times\cdots\times I_{n+1}$ is weakly $n$-absorbing and $I_{j}$'s different from $I_{i}$ are proper, then either
{\small $$
\begin{array}{ll}
&(1,\dots,1,\overbrace{a}^{t-th},1,\dots,1)(0,1,\dots,1)(1,0,1,\dots,1)
\cdots(1,\dots,1,0,\overbrace{1}^{i-th},\dots,1)\\
&(1,\dots,\overbrace{1}^{i-th},0,1,\dots,1)\cdots(1,\dots,1,0,\overbrace{1}^{t-th},\dots,1)(1,\dots,\overbrace{1}^{t-th},0,1,\dots,1)\\
&\cdots(1,\dots,1,0)=(0,\dots,0,\overbrace{1}^{i-th},0,\dots,0,\overbrace{a}^{t-th},0,\dots,0)\in L,
\end{array}
$$}
or
{\small $$
\begin{array}{ll}
&(1,\dots,1,\overbrace{b}^{t-th},1,\dots,1)(0,1,\dots,1)(1,0,1,\dots,1)
\cdots(1,\dots,1,0,\overbrace{1}^{i-th},\dots,1)\\
&(1,\dots,\overbrace{1}^{i-th},0,1,\dots,1)\cdots(1,\dots,1,0,\overbrace{1}^{t-th},\dots,1)(1,\dots,\overbrace{1}^{t-th},0,1,\dots,1)\\
&\cdots(1,\dots,1,0)=(0,\dots,0,\overbrace{1}^{i-th},0,\dots,0,\overbrace{b}^{t-th},0,\dots,0)\in L,
\end{array}
$$}
and thus either $a\in I_t$ or $b\in I_t$. Consequently $I_t$ is a prime ideal of $R_t$.\\
Now, assume that $L=I_{1}\times\dots\times I_{\alpha_{1}-1}\times R_{\alpha_{1}}\times I_{\alpha_{1}+1}\times\dots\times I_{\alpha_{j}-1}\times R_{\alpha_{j}}\times I_{\alpha_{j}+1}\times\cdots\times I_{n+1}$ in which $\{\alpha_{1},\dots,\alpha_{j}\}\subset\{1,\dots,n+1\}$.
Since $L$ is $n$-absorbing, then $I_{1}\times\dots\times I_{\alpha_{1}-1}\times I_{\alpha_{1}+1}\times\dots\times I_{\alpha_{j}-1}\times I_{\alpha_{j}+1}\cdots\times I_{n+1}$ is an $n$-absorbing ideal of $R_{1}\times\dots\times R_{\alpha_{1}-1}\times R_{\alpha_{1}+1}\times\dots\times R_{\alpha_{j}-1}\times R_{\alpha_{j}+1}\times\cdots\times R_{n+1}$, by Theorem \ref{T4}.\\
(3)$\Rightarrow$(1) If $L$ is one of the given two forms, then it is easily verified that
$L$ is an $n$-absorbing ideal of $R$, and hence $L$ is a weakly $n$-absorbing ideal of $R$. 
\end{proof}

\section{Rings with Property that all Proper Ideals are Weakly $n$-absorbing}

\begin{theorem}
Let $R$ be a ring and $n$ a positive integer such that every proper ideal of
$R$ is a weakly $n$-absorbing ideal of $R$. Then 
\begin{enumerate}
\item $dim(R)=0$.
\item $R$ has at most $n+1$ prime ideals that are pairwise comaximal, in particular, $R$ has at most $n+1$ maximal ideals.
\end{enumerate}
\end{theorem}
\begin{proof}
(1) Suppose that $dim(R)\geq 1$; so $R$ has prime ideals $P\subset Q$. Choose $x\in Q\backslash P$,
and let $I=x^{n+1}R$. Then $x^{n}\in I$, since $I$ is a weakly $n$-absorbing ideal of $R$ and $0\neq x^{n+1}\in I$.
The reminder is similar to the proof of \cite[Theorem 5.9]{AB1}.\\
(2) Suppose that $P_{1},\dots,P_{n+2}$ are prime ideals of $R$ that are pairwise comaximal. Let
$I=P_{1}\cdots P_{n+1}$. By \cite[Theorem 2.6]{AB1}, $I$ is not an $n$-absorbing ideal of $R$. Hence $I$ is a weakly
$n$-absorbing ideal of $R$ that is not an $n$-absorbing ideal of $R$. Thus $I^{n+1}=\{0\}$ by
Theorem \ref{nil}. Hence $I^{n+1}=P_{1}^{n+1}\cdots P_{n+1}^{n+1}=\{0\}\subseteq P_{n+2}$, and thus one of the $P_{i}$'s,
$1\leq i\leq n+1$, is contained in $P_{n+2}$, which is a contradiction. Hence $R$ has at most $n+1$ prime ideals that are pairwise comaximal.
\end{proof}

For a commutative ring $R$, we denote by $J(R)$ the intersection of all maximal ideals of $R$.

\begin{lemma}\label{T8}
Let $R$ be a commutative ring and $x_1,\dots,x_{n+1}\in J(R)$. Then the ideal $x_1\cdots x_{n+1}R$ is a weakly $n$-absorbing ideal of $R$ if and only if $x_1\cdots x_{n+1} = 0$.
\end{lemma}

\begin{proof}
Set $I = x_1\cdots x_{n+1}R$. If $x_1\cdots x_{n+1}= 0$, then $I$ is a weakly $n$-absorbing ideal of $R$. For the converse, assume that $I$ is a weakly $n$-absorbing ideal of $R$ and $x_1\cdots x_{n+1}\neq 0$. Since $x_1\cdots x_{n+1}\in I\setminus \{0\}$, then there are $n$ of $x_i$' whose product is in $I$. We may assume  that $y=x_1\cdots x_{n}\in I$. Hence $y=yx_{n+1}b$ for some $b\in R$ and so $y(1-x_{n+1}b)=0$. Since $x_{n+1}b\in J(R)$,  $1-x_{n+1}b$ is a unit of $R$. Therefore $y=0$ and then $x_1\cdots x_{n+1}=0$, which is a contradiction. Hence $x_1\cdots x_{n+1}=0$.
\end{proof}

\begin{corollary}\label{jac}
Let $R$ be a ring. If every proper ideal of $R$ is weakly $n$-absorbing, then $Jac(R)^{n+1}=0$, and so $Jac(R)=Nil(R)$.
\end{corollary}

\begin{theorem}\label{power}
Let $R$ be a semi-local ring with maximal ideals $M_{1},\dots,M_{t}$. If  for every nonnegative integers $\alpha_{1},\dots,\alpha_{t}$ with $\alpha_{1}+\cdots+\alpha_{t}=n+1$ we have $M_{1}^{\alpha_{1}}\cdots M_{t}^{\alpha_{t}}=\{0\}$,
then every proper ideal of $R$ is weakly $n$-absorbing.
\end{theorem}
\begin{proof}
Let $I$ be a proper ideal of $R$ and suppose that $0\neq x_{1}\cdots x_{n+1}\in I$ for some $x_{1},\dots,x_{n+1}\in R$. If for some 
$1\leq i\leq n+1$, $x_{i}$ is invertible, then $x_{1}\cdots\widehat{x_{i}}\cdots x_{n+1}\in I$. If for every $1\leq i\leq n+1$, $x_{i}$ is noninvertible, then there are nonnegative integers $\alpha_{1},\dots,\alpha_{t}$ with $\alpha_{1}+\cdots+\alpha_{t}=n+1$ such that $x_{1}\cdots x_{n+1}\in M_{1}^{\alpha_{1}}\cdots M_{t}^{\alpha_{t}}=\{0\}$, a contradiction. Consequently $I$ is weakly $n$-absorbing.
\end{proof}

As an immediate consequence of Corollary \ref{jac} and Theorem \ref{power} we have the next corollary.
\begin{corollary}
Let $(R,M)$ be a quasi-local ring. Then every proper ideal of $R$ is weakly $n$-absorbing if and only if $M^{n+1}=0$.
\end{corollary}

\begin{theorem}
Let $R$ be a semi-local ring with maximal ideals $M_{1},\dots,M_{t}$. If  for every nonnegative integers $\alpha_{1},\dots,\alpha_{t}$ with $\alpha_{1}+\cdots+\alpha_{t}=n$ we have $M_{1}^{\alpha_{1}}\cdots M_{t}^{\alpha_{t}}=\{0\}$,
then every proper ideal of $R$ is $n$-absorbing.
\end{theorem}
\begin{proof}
Let $I$ be a proper ideal of $R$ and suppose that $x_{1}\cdots x_{n+1}\in I$ for some $x_{1},\dots,x_{n+1}\in R$.  By Theorem \ref{power},
 $I$ is a weakly $n$-absorbing ideal of $R$. Hence if $x_{1}\cdots x_{n+1}\neq 0$, then we are done. Thus assume that $x_{1}\cdots x_{n+1}=0$. If for some 
$1\leq i\leq n$, $x_{i}$ is invertible, then $x_{1}\cdots\widehat{x_{i}}\cdots x_{n+1}=0\in I$. If for every $1\leq i\leq n$, $x_{i}$ is noninvertible, then there are nonnegative integers $\alpha_{1},\dots,\alpha_{t}$ with $\alpha_{1}+\cdots+\alpha_{t}=n$ such that $x_{1}\cdots x_{n}\in M_{1}^{\alpha_{1}}\cdots M_{t}^{\alpha_{t}}=\{0\}$. Consequently $I$ is $n$-absorbing.
\end{proof}

\begin{corollary}\label{power2}
Let $(R,M)$ be a quasi-local ring such that $M^{n}=\{0\}$. Then
every proper ideal of $R$ is $n$-absorbing.
\end{corollary}

\begin{theorem}\label{localrings}
Let $s>1$ be an integer, $(R_1, M_1),\dots,(R_s, M_s)$ be quasi-local commutative rings 
and let $R= R_1\times\cdots\times R_s$. If every proper ideal of $R$ is a weakly $n$-absorbing ideal of $R$, then $M_1^n =M_2^n=\cdots= M_s^n = \{0\}$. 
\end{theorem}

\begin{proof}
Assume that every proper ideal  of $R$ is a weakly $n$-absorbing ideal. Take an arbitrary integer $1\leq i\leq s$ and let $a_1,... ,a_n\in M_i$ such that $a_1\cdots a_n\neq 0$. Then $$I= \{0\}\times\cdots\times\{0\}\times(a_1\cdots a_n R_i)\times\{0\}\times \cdots\times \{0\}$$ is a weakly $n$-absorbing ideal of $R$. So we have 
\begin{eqnarray*}
(1,\dots,1,a_1,1,\dots,1)\cdots
(1,\dots,1,a_n,1,\dots,1)(0,\dots,0,1,0,\dots,0)\\
=(0,\dots,0,a_{1}\cdots a_{n},0,\dots,0)\in I\backslash\{(0,\dots,0)\}.
\end{eqnarray*}
Since $I$ is weakly $n$-absorbing, there exists $1\leq j\leq n$ such that
\begin{eqnarray*}
(1,\dots,1,a_1,1,\dots,1)\cdots
(1,\dots,1,a_{j-1},1,\dots,1)(1,\dots,1,a_{j+1},1,\dots,1)\\
\cdots(1,\dots,1,a_{n},1,\dots,1)(0,\dots,0,1,0,\dots,0)\in I.
\end{eqnarray*}
Then $a_1\cdots a_{j-1}a_{j+1}\cdots a_{n}=a_1\cdots a_nb$ for some $b\in R_i$. So $a_1\cdots a_{j-1}a_{j+1}\cdots \newline a_{n}(1-a_jb) = 0$. As $1-a_jb$ is a unit of $R_i$, we can conclude $a_1\cdots a_{j-1}a_{j+1}\cdots a_{n}\newline =0$, a contradiction. Thus for every $1\leq i\leq s$,  $M_i^n = \{0\}$. 
\end{proof}

\begin{theorem}\label{T9}
Let $(R_1, M_1)$ and $(R_2, M_2)$  be quasi-local commutative rings with $M_1^n = M_2^n = \{0\}$ and let $R= R_1\times R_2$. If  either $R_1$ or $R_2$ is a field, then every proper ideal of $R$ is a weakly $n$-absorbing ideal of $R$.
\end{theorem}

\begin{proof}
Let $R_2$ be a field. Since  $M_1^n = \{0\}$, so  every proper ideal of $R_1$ is an $n$-absorbing ideal, by Corollary \ref{power2}.  Thus, by Theorem \ref{T4} the ideal $\{0\} \times R_2$ is a weakly $n$-absorbing ideal of $R$. Since $R_2$ is a field, the ideal $R_1\times \{0\}$ is a weakly $n$-absorbing ideal of $R$. Now let $J$ be a proper ideal of $R_1$ such that $J\neq \{0\}$. Then $J$ is an $n$-absorbing ideal of $R_1$ and so $J\times R_2$ is a weakly $n$-absorbing ideal of $R$, by Theorem \ref{T4}. At last, we show that $I = J\times \{0\}$ is a weakly $n$-absorbing ideal of $R$. Assume that $(a_1,b_1)\cdots(a_{n+1},b_{n+1})\in I\setminus \{(0,0)\}$ such that $a_1,\dots,a_{n+1}\in R_1$ and $b_1,\dots,b_{n+1}\in R_2$. Since $0\neq a_{1}\cdots a_{n+1}\in J$ and $M_1^n = \{0\}$, then at least two of the $a_i$'s are not in $M_1$, say $a_{n}$ and $a_{n+1}$. Since $a_{n}$ and $a_{n+1}$ are unites of $R_1$ and $a_1\cdots a_{n+1}\in J$ we conclude that $a_1\cdots a_{n-1}\in J$. On the other hand, $R_2$ is a field and $b_1\cdots b_{n+1}=0$, at least one of the $b_i$'s is equal to $0$, say $b_{n+1}=0$. Hence  $(a_1,b_1)\cdots(a_{n-1},b_{n-1})(a_{n+1},0) \in I$. Therefore $I$ is a weakly $n$-absorbing ideal of $R$.
\end{proof}

\begin{theorem}\label{T10}
Let $R_1,R_2,\dots,R_{n+1}$ be commutative rings and let $R = R_1\times R_2\times \cdots \times R_{n+1}$. Then every proper ideal of $R$ is a weakly $n$-absorbing ideal of $R$ if and only if all of $R_i$'s  are fields.
\end{theorem}

\begin{proof}
Assume that every proper ideal of $R$ is a weakly $n$-absorbing ideal of $R$ and one of the $R_i$'s is not a field. Now we may assume that $R_1$ is not a field. Hence $R_1$ has a proper ideal. Say $J$ such that $J\neq \{0\}$. So the ideal $I = J\times \{0\}\times \cdots \times \{0\}$ of $R$ is a weakly $n$-absorbing ideal. Let $a\in J$ such that $a\neq 0$. Then 
\begin{eqnarray*}
(a,1,\dots,1) (1,0,1,\dots,1)(1,1,0,1,\dots,1)\cdots (1,\dots,1,0) \\
= (a,0,\dots,0)\in I\setminus \{(0,\dots,0)\}.
\end{eqnarray*}
Since $J$ is proper, there exists $2\leq i\leq n+1$ such that
$(a,0,\dots,0,\overbrace{1}^{i-{th}},0,\dots,0)\newline\in I$, which is a contradiction. Hence all of the $R_i$'s are fields. Conversely, suppose that $R$ is ring-isomorphic to $D = F_1\oplus F_2 \oplus \cdots \oplus F_{n+1}$ where $F_i$'s are fields. Since every nonzero proper ideal of $D$ is a product (intersection) of $m$ distinct maximal ideals of $D$, for some $1 \leq m \leq n$, we conclude that every nonzero proper ideal of $D$ is an $n$-absorbing ideal of $D$, \cite[Theorem 2.1]{AB1}. Hence every nonzero proper ideal of $R$ is a weakly $n$-absorbing ideal of $R$.
\end{proof}

\section{On Anderson-Badawi's Conjecture and Badawi-Yousefian's Question}

\begin{theorem}\label{main} 
Let $I$ be a proper ideal of a $u$-ring $R$. Then the following conditions are equivalent:\\
(a) $I$ is strongly $n$-absorbing;\\
(b) $I$ is $n$-absorbing;\\
(c) For every $t$ ideals $I_{1},\dots,I_{t}$, $0\leq t\leq n$, and for every elements $x_{1},\dots,x_{n-t}\in R$ such that  $x_{1}\cdots x_{n-t}I_{1}\cdots I_{t}\nsubseteq I$, 
\begin{eqnarray*}
(I:_{R}x_{1}\cdots x_{n-t}I_{1}\cdots I_{t})&=&[\cup_{i=1}^{n}(I:_{R}x_{1}\cdots\widehat{x_{i}}\cdots x_{n-t}I_{1}\cdots I_{t})]\\
&\cup&[\cup_{i=1}^{n}(I:_{R}x_{1}\cdots x_{n-t}I_{1}\cdots\widehat{I_{i}}\cdots I_{t})];
\end{eqnarray*}
(d) For every $t$ ideals $I_{1},\dots,I_{t}$, $0\leq t\leq n$, and for every elements $x_{1},\dots,x_{n-t}\in R$ such that $x_{1}\cdots x_{n-t}I_{1}\cdots I_{t}\nsubseteq I$, $$(I:_{R}x_{1}\cdots x_{n-t}I_{1}\cdots I_{t})=(I:_{R}x_{1}\cdots\widehat{x_{i}}\cdots x_{n-t}I_{1}\cdots I_{t})$$ for some $1\leq i\leq n-t$
or $$(I:_{R}x_{1}\cdots x_{n-t}I_{1}\cdots I_{t})=(I:_{R}x_{1}\cdots x_{n-t}I_{1}\cdots\widehat{I_{j}}\cdots I_{t})$$ for some $1\leq j\leq t$.
\end{theorem}
\begin{proof}
(a)$\Rightarrow$(b) It is clear.\\
(b)$\Rightarrow$(c) We use induction on $t$. For $t=0$, consider elements $x_{1},\dots,x_{n}\in R$ such that $x_{1}\cdots x_{n}\notin I$. We show that $$(I:_{R}x_{1}\cdots x_{n})=\cup_{i=1}^{n}(I:_{R}x_{1}\cdots\widehat{x_{i}}\cdots x_{n}).$$ Let $a\in(I:_{R}x_{1}\cdots x_{n})$, so $x_{1}\cdots x_{n}a\in I$. Since $x_{1}\cdots x_{n}\notin I$, then for some $1\leq i\leq n$ we have $x_{1}\cdots\widehat{x_{i}}\cdots x_{n}a\in I$, i.e., $a\in(I:_{R}x_{1}\cdots\widehat{x_{i}}\cdots x_{n})$. Therefore $$(I:_{R}x_{1}\cdots x_{n})\subseteq\cup_{i=1}^{n}(I:_{R}x_{1}\cdots\widehat{x_{i}}\cdots x_{n}).$$ 
Now suppose $t>0$ and assume that for integer $t-1$ the claim holds. Let $x_{1},\dots,x_{n-t}$ be elements of $R$ and let $I_{1},\dots,I_{t}$ be ideals of $R$ such that $x_{1}\cdots x_{n-t}I_{1}\cdots I_{t}\nsubseteq I$. Consider element $a\in(I:_{R}x_{1}\cdots x_{n-t}I_{1}\cdots I_{t})$. Thus $I_{t}\subseteq(I:_{R}x_{1}\cdots x_{n-t}aI_{1}\cdots I_{t-1})$. By hypothesis $(I:_{R}x_{1}\cdots x_{n-t}aI_{1}\cdots I_{t-1})\newline=(I:_{R}x_{1}\cdots\widehat{x_{i}}\cdots x_{n-t}aI_{1}\cdots I_{t-1})$ for some $1\leq i\leq n-t$ or 
$(I:_{R}x_{1}\cdots x_{n-t}aI_{1}\cdots I_{t-1})=(I:_{R}x_{1}\cdots x_{n-t}aI_{1}\cdots\widehat{I_{j}}\cdots I_{t-1})$ for some $1\leq j\leq t-1$. Consequently either $a\in(I:_{R}x_{1}\cdots\widehat{x_{i}}\cdots x_{n-t}I_{1}\cdots I_{t-1}I_{t})$ for some $1\leq i\leq n-t$ or $a\in(I:_{R}x_{1}\cdots x_{n-t}I_{1}\cdots\widehat{I_{j}}\cdots I_{t-1}I_{t})$ for some $1\leq j\leq t-1$. Hence 
\begin{eqnarray*}
(I:_{R}x_{1}\cdots x_{n-t}I_{1}\cdots I_{t})&=&[\cup_{i=1}^{n}(I:_{R}x_{1}\cdots\widehat{x_{i}}\cdots x_{n-t}I_{1}\cdots I_{t})]\\
&\cup&[\cup_{i=1}^{n}(I:_{R}x_{1}\cdots x_{n-t}I_{1}\cdots\widehat{I_{i}}\cdots I_{t})].
\end{eqnarray*}
(c)$\Rightarrow$(d) Since $R$ is a $u$-ring, we are done.\\
(d)$\Rightarrow$(a) In special case of part (d), for every ideals $I_{1},\dots,I_{n}$ of $R$ such that $I_{1}\cdots I_{n}\nsubseteq I$ we have $$(I:_{R}I_{1}\cdots I_{n})=(I:_{R}I_{1}\cdots\widehat{I_{i}}\cdots I_{n})$$ for some $1\leq i\leq n$. Now, easily we can see that $I$ is strongly $n$-absorbing.
\end{proof}

\begin{remark}
Note that in Theorem \ref{main}, for the case $n=2$ we can omit the condition $u$-ring, 
by the fact that if an ideal (a subgroup) is the union of two ideals (two
subgroups), then it is equal to one of them.
\end{remark}

In the next theorem we investigate weakly $n$-absorbing ideals over $u$-rings.
Notice that any B$\acute{\rm e}$zout ring is a $u$-ring, \cite[Corollary 1.2]{Q}.
\begin{theorem}\label{main}
Let $I$ be a proper ideal of a $u$-ring $R$. Then the following conditions are equivalent:\\
(a) $I$ is strongly weakly $n$-absorbing;\\
(b) $I$ is weakly $n$-absorbing;\\
(c) For every $t$ ideals $I_{1},\dots,I_{t}$, $0\leq t\leq n$, and for every elements $x_{1},\dots,x_{n-t}\in R$ such that  $x_{1}\cdots x_{n-t}I_{1}\cdots I_{t}\nsubseteq I$, 
\begin{eqnarray*}
(I:_{R}x_{1}\cdots x_{n-t}I_{1}\cdots I_{t})&=&[\cup_{i=1}^{n}(I:_{R}x_{1}\cdots\widehat{x_{i}}\cdots x_{n-t}I_{1}\cdots I_{t})]\\
&\cup&[\cup_{i=1}^{n}(I:_{R}x_{1}\cdots x_{n-t}I_{1}\cdots\widehat{I_{i}}\cdots I_{t})]\\
&\cup&(0:_{R}x_{1}\cdots x_{n-t}I_{1}\cdots I_{t});
\end{eqnarray*}
(d) For every $t$ ideals $I_{1},\dots,I_{t}$, $0\leq t\leq n$, and for every elements $x_{1},\dots,x_{n-t}\in R$ such that $x_{1}\cdots x_{n-t}I_{1}\cdots I_{t}\nsubseteq I$, $$(I:_{R}x_{1}\cdots x_{n-t}I_{1}\cdots I_{t})=(I:_{R}x_{1}\cdots\widehat{x_{i}}\cdots x_{n-t}I_{1}\cdots I_{t})$$ for some $1\leq i\leq n-t$
or $$(I:_{R}x_{1}\cdots x_{n-t}I_{1}\cdots I_{t})=(I:_{R}x_{1}\cdots x_{n-t}I_{1}\cdots\widehat{I_{j}}\cdots I_{t})$$ for some $1\leq j\leq t$ or $$(I:_{R}x_{1}\cdots x_{n-t}I_{1}\cdots I_{t})=(0:_{R}x_{1}\cdots x_{n-t}I_{1}\cdots I_{t}).$$
\end{theorem}
\begin{proof}
(a)$\Rightarrow$(b) It is clear.\\
(b)$\Rightarrow$(c) We use induction on $t$. For $t=0$, consider elements $x_{1},\dots,x_{n}\in R$ such that $x_{1}\cdots x_{n}\notin I$. We show that  $$(I:_{R}x_{1}\cdots x_{n})=\cup_{i=1}^{n}(I:_{R}x_{1}\cdots\widehat{x_{i}}\cdots x_{n})\cup(0:_{R}x_{1}\cdots x_{n}).$$ Let $a\in(I:_{R}x_{1}\cdots x_{n})$, so $x_{1}\cdots x_{n}a\in I$. Assume that $x_{1}\cdots x_{n}a\neq0$. Since $x_{1}\cdots x_{n}\notin I$, then for some $1\leq i\leq n$ we have $x_{1}\cdots\widehat{x_{i}}\cdots x_{n}a\in I$, i.e., $a\in(I:_{R}x_{1}\cdots\widehat{x_{i}}\cdots x_{n})$. Consequently $$(I:_{R}x_{1}\cdots x_{n})=[\cup_{i=1}^{n}(I:_{R}x_{1}\cdots\widehat{x_{i}}\cdots x_{n})]\cup(0:_{R}x_{1}\cdots x_{n}).$$ 
Now suppose $t>0$ and assume that for integer $t-1$ the claim holds. Let $x_{1},\dots,x_{n-t}$ be elements of $R$ and let $I_{1},\dots,I_{t}$ be ideals of $R$ such that $x_{1}\cdots x_{n-t}I_{1}\cdots I_{t}\nsubseteq I$. Consider element $a\in(I:_{R}x_{1}\cdots x_{n-t}I_{1}\cdots I_{t})$. Thus $I_{t}\subseteq(I:_{R}x_{1}\cdots x_{n-t}aI_{1}\cdots I_{t-1})$. By hypothesis $(I:_{R}x_{1}\cdots x_{n-t}aI_{1}\cdots I_{t-1})\newline=(I:_{R}x_{1}\cdots\widehat{x_{i}}\cdots x_{n-t}aI_{1}\cdots I_{t-1})$ for some $1\leq i\leq n-t$ or 
$(I:_{R}x_{1}\cdots x_{n-t}aI_{1}\cdots I_{t-1})=(I:_{R}x_{1}\cdots x_{n-t}aI_{1}\cdots\widehat{I_{j}}\cdots I_{t-1})$ for some $1\leq j\leq t-1$ or
$(I:_{R}x_{1}\cdots x_{n-t}aI_{1}\cdots I_{t-1})=(0:_{R}x_{1}\cdots x_{n-t}aI_{1}\cdots I_{t-1})$. Consequently either $a\in(I:_{R}x_{1}\cdots\widehat{x_{i}}\cdots x_{n-t}I_{1}\cdots I_{t-1}I_{t})$ for some $1\leq i\leq n-t$ or $a\in(I:_{R}x_{1}\cdots x_{n-t}I_{1}\cdots\widehat{I_{j}}\cdots I_{t-1}I_{t})$ for some $1\leq j\leq t-1$ or $a\in(0:_{R}x_{1}\cdots x_{n-t}I_{1}\cdots I_{t})$.
Hence 
\begin{eqnarray*}
(I:_{R}x_{1}\cdots x_{n-t}I_{1}\cdots I_{t})&=&[\cup_{i=1}^{n}(I:_{R}x_{1}\cdots\widehat{x_{i}}\cdots x_{n-t}I_{1}\cdots I_{t})]\\
&\cup&[\cup_{i=1}^{n}(I:_{R}x_{1}\cdots x_{n-t}I_{1}\cdots\widehat{I_{i}}\cdots I_{t})]\\
&\cup&(0:_{R}x_{1}\cdots x_{n-t}I_{1}\cdots I_{t}).
\end{eqnarray*}
(c)$\Rightarrow$(d) Since $R$ is a $u$-ring, we are done.\\
(d)$\Rightarrow$(a) In special case of part (d), for every ideals $I_{1},\dots,I_{n}$ of $R$ such that $I_{1}\cdots I_{n}\nsubseteq I$ we have $(I:_{R}I_{1}\cdots I_{n})=(I:_{R}I_{1}\cdots\widehat{I_{i}}\cdots I_{n})$ for some $1\leq i\leq n$ or
$(I:_{R}I_{1}\cdots I_{n})=(0:_{R}I_{1}\cdots I_{n})$. Now, easily we can see that $I$ is strongly weakly $n$-absorbing.
\end{proof}


\end{document}